%% file: Main.tex
\documentclass{elsarticle}
\input{Header}
\begin{document}

\input{Abstract}

\title{Examples and cofibrant generation of effective Kan fibrations}

\author[1]{Benno van den Berg}
\affiliation[1]{organization={Universiteit van Amsterdam}, country={The Netherlands}}

\author[2]{Freek Geerligs\fnref{fnFreek}}
\affiliation[2]{organization={Department of Computer Science and Engineering \\
University of Gothenburg and Chalmers University of Technology}, country={Sweden}}
\fntext[fnFreek]{This paper is a summary of 
the Masters' thesis \cite{MyThesis}, which was written while Freek 
was a student at Utrecht University.}

\maketitle

\section{Introduction}
\input{Intro}

\section{Notions of Kan fibrations }\label{sec:Definitions}
\input{NotionsOfKanFibrations}

\section{Simplicial Malcev algebras}\label{sec:Malcev}
\input{Malcev}

\section{Algebraic weak factorisation systems}\label{sec:LAWFS}
\input{LiftingAWFS}

\section{Outlook}
\input{Outlook}

\bibliography{References}
\bibliographystyle{alpha}

\end{document}

%% file: Header.tex
\usepackage[a4paper]{geometry} 
\usepackage[english]{babel} 
\usepackage{parskip} 
\usepackage{amsmath, amssymb, textcomp,amsthm} 
\usepackage{color} 
\usepackage{enumerate} 
\usepackage{hyperref} 
\usepackage{tikz-cd}
\usetikzlibrary{shapes,patterns,positioning}
\usepackage{tikz}
\usepackage{verbatim}

\usepackage{graphicx}
\usepackage{caption}
\usepackage{subcaption}
\usepackage{pdfpages}
\usepackage{glossaries}

\usepackage{mathrsfs}
\renewcommand{\gls}{\textbf}

\newcommand{\chapter}{\section}

\newtheorem{theorem}{Theorem}[section]

\newtheorem{lemma}[theorem]{Lemma}
\newtheorem{proposition}[theorem]{Proposition}
\newtheorem{corollary}[theorem]{Corollary}

\makeatletter
    \def\@endtheorem{\hfill $\lozenge$\endtrivlist\@endpefalse }
\makeatother
\theoremstyle{definition}

\newtheorem{example}[theorem]{Example}
\newtheorem{definition}[theorem]{Definition}

\newtheorem{remark}[theorem]{Remark}
\newtheorem{construction}[theorem]{Construction}

\newcommand{\RLP}{{\bf RLP}}
\newcommand{\LLP}{{\bf LLP}}

\newcommand{\cat}[1]{\ensuremath{\mathbf{#1}}}
\newcommand{\Sq}{\mathbb{S}\cat{q}}

%% file: Abstract.tex
\begin{abstract}
We will make two contributions to the theory of effective Kan fibrations, which are a more explicit version of the notion of a Kan fibration, a notion which plays a fundamental role in simplicial homotopy theory. We will show that simplicial Malcev algebras are effective Kan complexes and that the effective Kan fibrations can be seen as the right class in an algebraic weak factorisation system. In addition, we will introduce two strengthenings of the notion of an effective Kan fibration, the
symmetric effective and degenerate-preferring Kan fibrations, and show that the previous results holds for these strengthenings as well.
\end{abstract}

%% file: Intro.tex
In this paper we make several contributions to the development of simplicial homotopy theory, in particular the Kan-Quillen model structure, in an explicit or constructive style. Our motivation comes from homotopy type theory \cite{hottbook}, where Kan fibrations have been used in \cite{Voevodsky} to model dependent types for homotopy type theory. An essential part of this proof is that Kan fibrations are closed under pushforwards, so we can model dependent products. 

Homotopy type theory in turn models constructive mathematics.  However, in \cite{BCP} (see also \cite{FunctionalKan}), it has been shown that a proof that Kan fibrations 
are closed under pushforwards cannot be made constructively. This problem has been addressed by moving away from simplicial sets,
and instead use cubical sets, as in \cite{CCHM} and \cite{BCH}. However, as simplicial sets are important and pervasive in homotopy theory, 
it is still desirable to have a constructive notion of Kan fibrations in simplicial sets. 

At this point in time, there are two notions of Kan fibrations in simplicial sets which can 
constructively be shown to be closed under pushforwards, namely the effective Kan fibrations from \cite{BergFaber} and the uniform Kan fibrations from \cite{GambinoSattler}. We will not discuss uniform Kan fibration here, but let us note that this notion of fibration is not ``local'', as discussed in Appendix D of \cite{BergFaber}, and every effective Kan fibration is also a uniform Kan fibration, as shown in Corollary 11.2 of \cite{BergFaber}. 

Therefore we will focus on the effective Kan fibrations, which are Kan fibrations equipped with explicit solutions to lifting problems against horns satisfying a certain compatibility condition. The hope is that for these effective Kan fibrations we can obtain constructive and explicit versions of many standard results in simplicial homotopy theory, as was done in \cite{BergFaber} for closure under pushforwards. In this paper we will show that that is indeed the case for two other standard results. 

The first results says that simplicial groups are Kan complexes \cite{Moore}. The proof of this result is constructive, in that the argument constructs explicit lifts against horn inclusions. We will show that these lifts satisfy the additional compatibility requirement making them effective Kan complexes. In fact, we will show that all simplicial Malcev algebras are effective Kan complexes, where Malcev algebras are 
exactly those algebras for which all simplicial objects are Kan complexes, 
as shown in \cite{CKP} and 
\cite{KanFibrationsForMalcevAlgebras}.
We will describe these results in Section \ref{sec:Malcev}.

The second result is that Kan fibrations are the right class in a weak factorisation system. By now there is a well-established explicit version of a weak factorisation system, which goes by the name of an ``algebraic weak factorisation system'' \cite{GarnerSmallObjects,BourkeGarner,Bourke}. In Section \ref{sec:LAWFS} we will show that the effective Kan fibrations are the right class in an algebraic weak factorisation system.

In Section \ref{sec:Definitions} we will also introduce two strengthenings of the notion of an effective Kan fibration, the \text{symmetric effective Kan fibrations} and the 
\text{degenerate-preferring Kan fibrations}, respectively. We will show that our two main results hold for these classes of maps as well.

%% file: NotionsOfKanFibrations.tex
In this section, we shall recall the necessary mathematical background on Kan fibrations. 
Then we shall introduce the newer notions of Kan fibrations. 
\subsection{Mathematical background on Kan fibrations}
In this subsection, we shall recall some concepts from higher category theory.
We assume some 
familiarity with category theory and topos theory.
References 
include
\cite{kerodon},
\cite{RiehlCat} 
and 
\cite{PTJ}.

For each natural number $n\in\mathbb N$, we define 
the category $[n]$ as the poset $\mathbb N_{\leq n}$, 
ordered under $\leq$. 
Functors between the category $[n]$ and $[m]$ are then order-preserving functions
between $\mathbb N_{\leq n}$ and $\mathbb N_{\leq m}$.
The category with the objects
$\{[n] \, | \, n\in\mathbb N\}$ and morphisms the order-preserving functions is denoted 
$\Delta$ and called the \textbf{simplex category}.
The simplex category has two special classes of maps, called 
the {degeneracy} and {face maps}. 
\begin{itemize}
        \item 
                For any $n\in\mathbb N$, $0\leq i\leq n$, we have a 
                \textbf{degeneracy map} ${s_i}$
                hitting $i$ twice. 
                \begin{equation}
                        s_i : [n+1]\to [n], \hspace{1cm}
                        s_i(k) = \begin{cases}
                        k \text{ if } k \leq i \\
                        k-1 \text{ if } k > i  \end{cases}
                \end{equation}

        \item 
                For any $n\in\mathbb N$, $0\leq i \leq n$, we have a 
                \textbf{face map} ${d_i}$
                skipping over $i$. 
                \begin{equation}
                d_i : [n]\to [n+1], \hspace{1cm}
                d_i(k) = \begin{cases}
                k \text{ if } k < i \\
                k+1 \text{ if } k \geq i
                \end{cases}
                \end{equation}
\end{itemize}
Every map in $\Delta$ can be written 
as a composition of face and degeneracy maps. 
Between the face and degeneracy maps, we have the following composition laws, 
called the \textbf{simplicial identities}:
\begin{align}
        s_j\circ d_k &= \begin{cases}
                d_{k-1}\circ s_j &\text{ if } k>j+1 \\
                1                &\text{ if } k\in\{j,j+1\}\\
                d_k\circ s_{j-1} &\text{ if } k<j \\
        \end{cases} \\
        d_j\circ d_k &= d_{k+1}\circ d_j    \text{ if } k \geq j \\
        s_j\circ s_k &= s_{k} \circ s_{j+1} \text { if } j\geq k
\end{align}

A \textbf{simplicial set} is 
a presheaf on $\Delta$. 
More generally, for any category $\mathcal C$, a simplicial object of 
$\mathcal C$ is a functor  $\Delta^{op} \to \mathcal C$. 
The category of simplicial sets is denoted \textbf{sSet}.
The representable presheaf corresponding to $[n]$ is 
denoted ${\Delta^n}$ and called the $n$-simplex. 
For a simplicial set $G$, the Yoneda lemma allows us to identify maps $\Delta^n\to G$
with elements of $G_n$.

As $\textbf{sSet}$ is a presheaf category, 
it is a topos and the subobjects
of any simplicial set carry a lattice structure. 
For any two subobjects $A\hookrightarrow C, B\hookrightarrow C$, 
the intersection $A\cap B$
is given by taking the pullback of the two arrows, while the union $A\cup B$ is given 
by taking the pushout of the arrows 
$A\cap B \hookrightarrow  A$ and $A\cap B \hookrightarrow B$. 
For all $n>0$ and each $0\leq m \leq n$, we define the following subobject of $\Delta^n$:
\begin{equation}
        \Lambda^n_m = \bigcup_{\substack{0\leq k \leq n\\ k \neq m }} d_k
\end{equation}
$\Lambda^n_m$ is called a \textbf{horn}, and the mono 
$\iota: \Lambda^n_m \hookrightarrow \Delta^n$ is called a \textbf{horn inclusion}.
We have abused notation by writing $d_k$
for the subobject of $\Delta^{n-1}\hookrightarrow \Delta^n$ corresponding to the mono 
$d_k:[n-1]\hookrightarrow [n]$.
We will also use $d_k$ for the map $\Delta^{n-1}\hookrightarrow \Lambda^n_m$ if $k\neq m$.

A map $\alpha:X\to Y$ of simplicial sets is a \textbf{Kan fibration} if 
for any horn inclusion $\iota:\Lambda^n_m \hookrightarrow \Delta^n$
and any pair of maps $x,y$ making the solid part of the following diagram commute, 
there exists a dashed map, called a \textbf{lift}, making the entire diagram commute. 
        \begin{equation}\label{eqn:dfnKanFib}\begin{tikzcd}
                \Lambda^n_m \arrow[d,hook,"\iota"'] \arrow[r,"x"] & X \arrow[d,"\alpha"] \\
                \Delta^n \arrow[r, "y"] \arrow[ru,dashed] & Y
        \end{tikzcd}\end{equation}
A simplicial set $X$ is called a \textbf{Kan complex} 
if the unique map to the terminal object $X\to \Delta^0$ is a Kan fibration.
We will use the same vocabulary for the other notions of Kan fibration introduced 
in this paper.

\subsection{Degenerate-preferring Kan fibrations}
In this subsection, we shall introduce the degenerate-preferring Kan fibrations, which are examples of functional Kan fibrations.

\begin{definition}\label{dfnFuncKanFib}
        A \textbf{functional Kan fibration} is a morphism of 
        simplicial sets $\alpha:X\to Y$, 
        together with a function $\text{lift}$, which takes as input maps 
        $x,y$ making the solid diagram in Diagram \ref{eqn:dfnKanFib} commute, 
        and outputs a morphism $\text{lift}(x,y):\Delta^n\to X$  
        making the diagram commute. 
\end{definition}

The definition of the degenerate-preferring Kan fibrations can be motivated by the following result of Van den Berg, Faber and Sattler. In Appendix C of \cite{BergFaber} they show that
if the lifting problem 
in Diagram \ref{eqn:dfnKanFib} has a degenerate solution, it is unique:
\begin{lemma}
        Suppose that both $z\circ s_j$ and $z'\circ s_{j'}$ are possible lifts 
        for lifting problems as in Diagram \ref{eqn:dfnKanFib}. 
        Then $z\circ s_j = z'\circ s_{j'}$.
\end{lemma}
Using the law of excluded middle and the axiom of choice, any Kan 
fibration can be given the structure of a functional Kan fibration
satisfying the following definition:
\begin{definition}
        A \textbf{degenerate-preferring Kan fibration} is functional Kan fibration 
        for which we have that 
        $\text{lift}$ picks out a solution of the form 
        $z\circ s_j$ whenever this is possible. 
\end{definition}

\begin{remark}
  Evan Cavallo has pointed out to us that degenerate-preferring Kan fibrations are similar to 
  the ``regular" or ``normal" Kan fibrations in cubical sets 
  (see, for example, the introduction of \cite{swanseparating} or \cite[Section 7]{gambinolarrea23}).
  Thierry Coquand based this name on the 
  regular fiber spaces, which first appeared in \cite{HurewiczFiber}.
\end{remark}

  \subsection{Symmetric effective Kan fibrations}
  \label{sec:SymEffKanFib}
In this section, we shall define symmetric effective Kan fibrations. 
This notion is based on the definition of effective Kan fibrations, 
which we shall discuss in Section \ref{sec:EffKanCom}.

Effective Kan fibrations are functional Kan fibrations as in Definition 
\ref{dfnFuncKanFib} where the function $\text{lift}$ satisfies certain conditions. 
Inspired by Theorem 12.1 from \cite{BergFaber}, 
these conditions amount to 
stability along pullback along degeneracy maps.
We will first give an example of what this means for a specific
lifting problem, and then give the general definition. 
\begin{example}
\input{ExampleSymEffKanCom}
\end{example}
For the general case, consider 
Diagram \ref{eqn:dfnKanFib}, 
and suppose we pull back $\iota$ along $s_j:\Delta^{n+1}\to \Delta^n$:
\begin{equation}
\begin{tikzcd}
	\arrow["\lrcorner"{anchor=center, pos=0.125}, draw=none, rd]
s_{j}^*(\Lambda^n_m) \arrow[r,"\sigma_{j}"]\arrow[d,hook] & 
        \Lambda^n_m \arrow[d,hook,"\iota"] \arrow[r,"x"] & X\arrow[d,"\alpha"]\\
\Delta^{n+1} \arrow[r,"s_{j}"] 
        & \Delta^n \arrow[r,"y"] \arrow[ru,dashed]& Y
\end{tikzcd}
\end{equation}

We shall argue that the inclusion 
$s_{j}^*(\Lambda^n_m)\hookrightarrow \Delta^{n+1}$ 
factors through a horn inclusion 
$\iota^*:\Lambda^{n+1}_{m^*} \hookrightarrow \Delta^{n+1}$. 
We shall then show that a dashed arrow as above allows us
to extend $x\circ\sigma_j$
to a map 
$s_j^*(x):\Lambda^{n+1}_{m^*} \to X$. 

By the pullback property, we have that $d_k$ factors through $s_{j}^*(\Lambda^n_m)$ iff
$s_{j}\circ d_k$ factors through $\Lambda^n_m$. 
\begin{itemize}
        \item
As $s_{j}\circ d_{j} = s_{j} \circ d_{j+1}=id$,
the faces
$d_{j}$ and $d_{j+1}$ do not factor through $s_j^*(\Lambda^n_m)$. 

\item 
If $j\neq m$, we have for 
    \begin{equation}\label{eqnm*}
	(j^*,m^*) = 
	\begin{cases}
		(j-1,m) \text { if }   m < j  \\
		(j,m+1) \text{ if }  m > j
	\end{cases}
\end{equation}
that
$s_{j}\circ d_{m^*}=d_m\circ s_{j^*}$. 
Now $d_m\circ s_{j^*}$ does not 
factor through $\Lambda^n_m$. 
So if $j\neq m$, we have that $d_{m^*}$ does not factor through $s_j^*(\Lambda^n_m)$. 

\item 
For all other faces $d_k$, we have that $s_{j}\circ d_k = d_{k'}\circ s_{{j}'}$
for some $k'\neq m$ and some $j'$, hence those $d_k$ do factor through $s_{j}(\Lambda^n_m)$. 
\end{itemize}

While $d_{j}, d_{j+1}$ do not factor through $s_{j}^*(\Lambda^n_m)$, 
the faces do intersect $s_{j}^*(\Lambda^n_m)$. 
And we do know what this intersection looks like.
If we use $d_{j_\pm}$ as shorthand for both $d_j$ and $d_{j+1}$, we can see that 
$s_j\circ d_{j_\pm} = id$. 
Since taking intersection with $d_{j_\pm}$ corresponds to 
taking a pullback along $d_{j_\pm}$,
and the pullback of the identity is always the identity, we get the following
diagram:

\begin{equation}
\begin{tikzcd}
\Lambda^n_m \arrow[r,dotted] \arrow[d,hook]
\arrow["\lrcorner"{anchor=center, pos=0.125}, draw=none, rd]
\arrow[rr, bend left = 40, "id"]
& s_j^*(\Lambda^n_m) \arrow[d,hook] 
\arrow["\lrcorner"{anchor=center, pos=0.125}, draw=none, rd]
\arrow[r,"\sigma_{j}"]
& \Lambda^n_m 
        \arrow[r,"x"] 
        \arrow[d,hook] 
        & X\arrow[d,"\alpha"]
\\
\Delta^n\arrow[r,"d_{j_\pm}"] 
\arrow[rr,bend right = 30, "id"']
& \Delta^{n+1} 
        \arrow[r,"s_{j}"]  & \Delta^n \arrow[r,"y"]& Y
\end{tikzcd}
\end{equation}

If there exists a map $\text{lift}(x,y):\Delta^n\to X$ making the right-most square 
above commute, 
we could extend $x\circ\sigma$ to a map $s_j^*(\Lambda^n_m)\cup d_{j_\pm} \to X$. 
which has $\text{lift}(x,y)$ as value on the face $d_{j_\pm}$. 
As the intersection of $d_{j_\pm}$ and $s_j^*(\Lambda^n_m)$ is given by $\Lambda^n_m$, 
where $\text{lift}(x,y)$ equals $x$, this is properly defined. 
This allows us to extend $x\circ\sigma$ to 
all of $\Lambda^{n+1}_{m^*}$.
\begin{itemize}
        \item
If $m\neq j$, then we missed three faces in $s_{j}^*(\Lambda^n_m)$, and if we 
add both the faces $d_{j},d_{j+1}$ as above, we get a new horn 
                $\Lambda^{n+1}_{m^*}$, and an extension $s_j^*(x):\Lambda^{n+1}_{m^*}\to X$ 
                of $X$. 
\item
If $m=j$, then we only miss two faces, namely $m,m+1$, 
and we need only add one face to get a new horn.
So we get two possible new horns, 
one where we add $d_m$, and one where we add $d_{m+1}$.
                We thus also have two possible extensions $s_j^*(x)$. 
\end{itemize}
So to summarize, if we have a dashed arrow $\text{lift}(x,y)$ in 
Diagram \ref{eqn:dfnKanFib}, for all $m^*$ with 
\begin{equation}\label{eqn:possibleNewMissingFaces}
	m^*\in \begin{cases}
		\{m\} \text { if } m<j\\
		\{m,m+1\} \text{ if } m =j\\
		\{m+1\} \text{ if }  m > j
	\end{cases}
\end{equation}
we define a map $s_{j}^*(x): \Lambda^{n+1}_{m^*}\to X$ 
with the following values on its faces: 
\begin{equation}\label{eqn:pulledbackHornmap}
s_{j}^*(x)\circ d_k = 
\begin{cases}
x\circ s_{j}\circ d_k\text { if } k\neq j,j+1,m^*\\
\text{lift} (x,y) \text{ if } k \in \{j,j+1\}- \{m^*\}
\end{cases}
\end{equation}
\begin{definition}
	\label{dfn:SymEffKanCom}
        A \gls{symmetric effective Kan fibration} 
        is a functional Kan fibration $(\alpha,\text{lift})$ such that 
        for any
	$0\leq j \leq n$ and any $m^*$ and $s_{j}^*(x)$ as described above, we have 
  \begin{equation}\label{eqndfnsymeffKancom}
	\text{lift}(s_{j}^*(x),y\circ s_j) = \text{lift}(x,y)\circ s_{j}.
  \end{equation}
\end{definition}

\begin{remark}
        There is another formulation of $s_j^*(x)$ suggested by Storm Diephuis:
        \begin{equation}\label{eqnStormsForm}
                s_j^*(x) = \text{lift}(x,y)\circ s_j\circ \iota^*
        \end{equation}
        Here $\iota^*$ is the inclusion
        $\Lambda^{n+1}_{m^*}\hookrightarrow \Delta^{n+1}$.
\end{remark}
\begin{remark}
        Note that $\text{lift}(x,y)\circ s_j$ is always a possible lift for 
        the lifting problem given by $(s_j^*(x),y\circ s_j)$. 
        As this lift is degenerate, a degenerate-preferring Kan fibration will 
        always choose it. Therefore any degenerate-preferring Kan fibration
        is a symmetric effective Kan fibration. 

        However, the converse is not true.
        Let $(X,\text{lift})$ be a degenerate-preferring Kan complex
        with a non-degenerate $1$-cell $A\to B$ in $X$.
        Let $a:\Delta^0\to X$ pick out $A$. 
        For $e$ the isomorphism $\Lambda^1_1\simeq \Delta^0$ consider
        $a\circ e$ as lifting problem against $\Lambda^1_1\hookrightarrow \Delta^1$.
        As $\text{lift}$ is degenerate-preferring, it 
        must pick out $a\circ s_0$ as lift. 
        Note that $a\circ e$ cannot be of the form $s_j^*(x)$:
        the domain of $x$ would have to be a $0$-horn, which does not exist. 
        So if we define $\text{lift}'$ by picking out $A\to B$ as lift for $a\circ e$ 
        and acting the same as $\text{lift}$ on all other lifting problems, 
        $(X,\text{lift}')$ will be a symmetric effective Kan complex, 
        but not degenerate-preferring. 
\end{remark}

\subsection{Effective Kan fibrations}\label{sec:EffKanCom}
        In this section we recall the definition of an 
        effective Kan fibrations. The definition we give here is based on Theorem 12.1 in \cite{BergFaber}. 
        We give a slightly different presentation here which is more in line with the
        rest of this paper. 
\begin{definition}
        A \textbf{signed horn inclusion} $(\iota,\pm)$ is a horn inclusion 
$\iota:\Lambda^n_m\hookrightarrow \Delta^n$, 
together with a sign $\pm\in \{+,-\}$,
such that 
if $m=0$, then $\pm = -$, and if $m=n$, then $\pm = +$.
\end{definition}

\begin{definition}
A \textbf{functional signed Kan fibration} is a morphism of 
        simplicial sets $\alpha:X\to Y$, 
        together with two functions $\text{lift}_+$ and $\text{lift}_-$.
        For $\pm\in\{+,-\}$, $\text{lift}_\pm$
        takes as input a signed horn inclusion $(\iota,\pm)$
        and a pair of maps 
        $x:\Lambda^n_m\to X, y:\Delta^n\to Y$ 
        as in Diagram \ref{eqn:dfnKanFib}.
        and outputs a morphism $\text{lift}_{\pm}(x,y):\Delta^n\to X$ 
        making the diagram commute. 
\end{definition}

\begin{definition}\label{dfnEffKanFib}
        An \textbf{effective Kan fibration} is functional signed Kan fibration 
        such that
        \begin{equation}\label{eqndfnEffKanFib}
	\text{lift}_\pm(
                \text{lift}_\pm(x,y)\circ s_j\circ \iota^*,y\circ s_j
                ) = 
                \text{lift}_\pm(x,y)\circ s_{j}
        \end{equation}
        for any $\iota,x,y$ as in Diagram \ref{eqn:dfnKanFib}, 
        any $\pm \in \{+,-\}$, any $0\leq j \leq n$, 
        any $m^*$ as in Equation \ref{eqn:possibleNewMissingFaces}, 
        and for $\iota^*$ the inclusion 
        $\Lambda^{n+1}_{m^*}\hookrightarrow \Delta^{n+1}$, 
\end{definition}
\begin{remark}
        Recall that the horn $\Lambda^n_m$ is called an inner horn if $m\notin \{0,n\}$, 
        and an outer horn if $m\in \{0,n\}$. 
        Effective Kan fibrations thus have two choices of lifts on inner horns, 
        and one choice of lifts on outer horns. 

        Note that if $\text{lift}_+$ and $\text{lift}_-$ agree on inner horns 
        for some effective Kan fibration, 
        they uniquely determine the structure of a symmetric effective Kan fibration. 
        Conversely, every symmetric effective Kan fibration can be used to 
        create an effective Kan fibration such that $\text{lift}_+$ and 
        $\text{lift}_-$ agree on inner horns. 
\end{remark}

\begin{remark}
The argument in Corollary 12.2 of \cite{BergFaber} 
shows that degenerate-preferring Kan fibrations must be 
effective Kan fibrations. In Section \ref{sec:Malcev}, we shall show constructively 
that a large class of Kan fibrations
can constructively be given the structure of degenerate-preferring Kan fibrations. 
Therefore, they are also effective Kan fibrations. 
\end{remark}

%% file: ExampleSymEffKanCom.tex
Consider in Equation \ref{eqn:dfnKanFib} the case where 
$\iota:\Lambda^2_2\hookrightarrow \Delta^2$, 
and suppose that $Y=\Delta^0$, the terminal object in $\gls{sSet}$.
Then $y,\alpha$ are uniquely determined.
For $x:\Lambda^2_2\to X$, we draw 
$x\circ d_0$ in green, $x\circ d_1$ in blue and $x\circ d_1\circ d_1$ in red. 
\input{triangle}

Assume now that $\alpha:X\to 1$ has the structure of a functional Kan fibration, 
so $X$ has the structure of a functional Kan complex. 
Then there exists an extension of $x$ 
to a map $\text{lift}(x,y):\Delta^2\to X$,
which we draw as follows:
\input{filledtriangle}

Suppose that we pull back 
$\iota$ along $s_0:\Delta^3\to \Delta^2$ and get 
a subobject $s_0^*(\Lambda^2_2)\hookrightarrow \Delta^3$. 
For $\sigma$ the map $s_0^*(\Lambda^2_2)\to \Lambda^2_2$, 
we want to draw $x\circ \sigma$.

Note that $s_0$ sends $0,1$ both to $0$, 
	sends $2$ to $1$ and $3$ to $2$. 
This gives that the pullback of $0$ is the arrow $0\to 1$,
the pullback of $1$ is $2$ and the pullback of 
$2$ is $3$.
Geometrically, we can interpret pulling back along $s_0$ 
as taking the red point
and stretching it out to a red line. We thus draw $x\circ\sigma$ as follows:
\input{tetrahedonSieve}

Note that when we stretch out the red point, 
we also stretch out a face of the blue arrow. 
As a consequence, the entire blue arrow should be stretched out along that face as well. 
This corresponds to stating that $x\circ\sigma\circ d_2$ should be equal to $x\circ d_1\circ s_0$.

Now suppose we want to extend $x\circ\sigma$ to a map $\Delta^3\to X$. 
Note that the faces $d_0$ and $d_1$ of $x\circ\sigma$ look exactly like $x$, 
as their boundary contains 
the original horn (the blue and green arrows). 
We can thus extend $x\circ\sigma$ to a map $s_0^*(x):\Lambda^3_3\to X$ by 
putting $s_0^*(x)\circ d_0=s_0^*(x)\circ d_1 = \text{lift}(x,y)$.
We can draw $s_0^*(x)$ as follows:
\input{Horn}

As $\alpha$ is a functional Kan fibration, 
it should have an extension for $s_0^*(x)$ to a map $\Delta^3\to X$. 
Note that $\text{lift}(x,y)\circ s_0$ should be a possible extension. 
The condition for $\alpha$ being an symmetric effective Kan fibration 
will be that we 
choose exactly $\text{lift}(x,y)\circ s_0$ 
as lift against $s_0^*(x)$. 

%% file: triangle.tex
\begin{equation}
\begin{tikzpicture}
\node (0) at (0,0) {\textcolor{red}{$0$}};
\node (1) at (2,0) {{$1$}};
\node (2) at (1,-1) {{$2$}};
\draw[->,blue] (0)--(2);
\draw[->,green] (1)--(2);
\end{tikzpicture}
\end{equation}

%% file: filledtriangle.tex
\begin{equation}
\begin{tikzpicture}
\shade[left color = {blue!30}, right color = {green!30}, opacity = 0.5 ] (2,0)--(1,-1)--(0,0)--cycle;
\node (0) at (0,0) {\textcolor{red}{$0$}};
\node (1) at (2,0) {{$1$}};
\node (2) at (1,-1) {{$2$}};
\draw[->,blue] (0)--(2);
\draw[->,green] (1)--(2);
\end{tikzpicture}
\end{equation}

%% file: tetrahedonSieve.tex
\begin{equation}
\begin{tikzpicture}
\fill[blue!10] (0,0)--(2,0)--(3,-1)--cycle;
\node (0) at (0,0) {\textcolor{red}{$0$}};
\node (1) at (2,0) {\textcolor{red}{$1$}};
\node (2) at (1,-1) {{$2$}};
\node (3) at (3,-1) {{$3$}};
\draw[red,->] (0)--(1);
\draw[blue, ->] (0)--(3);
\draw[blue, ->] (1)--(3);
\draw[green, ->] (2)--(3);
\end{tikzpicture}
\end{equation}

%% file: Horn.tex
\begin{equation}
\begin{tikzpicture}
\fill[blue!10] (0,0)--(2,0)--(3,-1)--cycle;
\shade[left color = {blue!30}, right color = {green!30}, opacity = 0.5 ] (2,0)--(1,-1)--(3,-1)--cycle;
\shade[left color = {blue!30}, right color = {green!30}, opacity = 0.5 ] (0,0)--(1,-1)--(3,-1)--cycle;
\node (0) at (0,0) {\textcolor{red}{$0$}};
\node (1) at (2,0) {\textcolor{red}{$1$}};
\node (2) at (1,-1) {{$2$}};
\node (3) at (3,-1) {{$3$}};
\draw[red] (0)--(1);
\draw[blue, ->] (0)--(3);
\draw[blue, ->] (1)--(3);
\draw[green, ->] (2)--(3);
\end{tikzpicture}
\end{equation}

%% file: Malcev.tex
A well-known result of Moore \cite[Theorem 3.4]{Moore} says  that simplicial groups are Kan complexes. 
Since then, this result has been generalized.
In \cite{CKP} it was shown that we can classify all algebras for which the 
simplicial objects are Kan complexes, namely the Malcev algebras.
In addition, in \cite{KanFibrationsForMalcevAlgebras} it is shown that any 
surjection of simplicial Malcev algebras is a Kan fibration. 

In this section, we shall generalize the latter result to degenerate-preferring Kan fibrations. 
First we shall give some context for simplicial Malcev algebras and introduce 
some new terminology. 

\subsection{Malcev algebras}\label{sec:MalcevMathContext}
Recall that an algebraic theory is a theory $\mathcal T$ where the axioms contain only 
equalities between terms. 
A simplicial algebra for $\mathcal T$ is a functor $\Delta^{op}\to\mathbb T$, 
where $\mathbb T$ is the category of models for $\mathcal T$.

The theory of Malcev algebras is such an algebraic theory.
It is written in the language with a ternary operation $\mu$ 
and has the axioms that $\mu(x,x,y)=y$ and $\mu(x,y,y)=x$. 
Such $\mu$ is called a \textbf{Malcev operation}. 
If an algebraic theory has a Malcev operation, it is called a \textbf{Malcev theory}. 
\begin{example}\label{exampleGroupIsMalcev}
The theory of groups is a Malcev theory with $\mu(x,y,z)=xy^{-1}z$. 
  The theory of Heyting algebras is a Malcev theory with 
$    \mu(x,y,z)=((z\to y)\to x)\wedge ((x\to y)\to z).$
\end{example}
In \cite[Proposition 1]{KanFibrationsForMalcevAlgebras} a concise 
and constructive proof
is given that implies the following:
\begin{proposition}\label{thm:AllSimpObjectsKanImpliesMalcevTheory}
        Let $\mathcal T$ be any algebraic theory. 
        If all simplicial algebras of $\mathcal T$ are Kan complexes,
        then $\mathcal T$ is a Malcev theory. 
\end{proposition}
So if we have a constructive proof that all simplicial algebras of $\mathcal T$ are Kan,
we have a constructive proof that $\mathcal T$ is Malcev. 
The converse of this theorem is also known to be true, and was even generalized further in 
\cite[Theorem 3]{KanFibrationsForMalcevAlgebras} to the following:
\begin{theorem}\label{thm:SurjectionMalcevKan}
Any surjective morphism $f:X\to Y$ of simplicial Malcev algebras is a Kan fibration. 
\end{theorem}

\subsection{A degenerate-preferring structure}
In this section, we will show a variation of Theorem \ref{thm:SurjectionMalcevKan}
holds for degenerate-preferring Kan fibrations. 
However, being a surjection turns out 
to be too weak of a notion for a constructive proof. 
Therefore we introduce the following definition: 
\begin{definition}\label{dfn:degeneracy-section}
  Let $\alpha:X\to Y$ be a map of simplicial sets. 
  Let $\beta = (\beta_n: Y_n\to X_n)_{n\in\mathbb N}$ be a collection 
  of functions. 
  We call $\beta$ a \gls{degeneracy-section} of $\alpha$ iff
  \begin{itemize}
    \item for all $n\in\mathbb N$ we have $\alpha_n\circ \beta_n = 1_{Y_n}$.
    \item for all $n\in\mathbb N$ and all $0\leq k\leq n$, 
            we have $\beta_{n+1}(y\circ s_{k}) = \beta_n(y)\circ s_k$. 
  \end{itemize}
\end{definition}

The variation of Theorem \ref{thm:SurjectionMalcevKan} we will show is as follows:
\begin{theorem}\label{thmMainMalcevThm}
  Let $\alpha:X\to Y$ be a morphism of simplicial algebras that has a degeneracy-section. 
  Then $\alpha$ is a degenerate-preferring Kan fibration. 
\end{theorem}

To show the above theorem, we will first recall in Construction 
\ref{constructionMalcevLiftsHelpers} the lifting method that 
\cite{KanFibrationsForMalcevAlgebras} uses to solve the required lifting problem. 
We will then show two properties of this lifting structure and conclude it is 
degenerate-preferring. 

  \begin{construction}\label{constructionMalcevLiftsHelpers}
Let $\alpha:X\to Y$ be a morphism of simplicial Malcev algebras with degeneracy-section
$(\beta_n:Y_n\to X_n)_{n\in\mathbb N}$.
Now consider maps $x,y$ as in 
  Diagram \ref{eqn:dfnKanFib}.

    We shall define for all $k\in \mathbb Z$ with $-1\leq k<m$ or $m<k\leq n+1$ 
    a helper function $w_k:\Delta^n \to X$. 
    We start by defining $w_{-1}=\beta_n(y)$. 

    Now for $k$ as above, we define
    \begin{equation}
      k' = \begin{cases}
        k     \text{ if } k < m \\
        k - 1 \text{ if } k > m \\
      \end{cases}
    \end{equation}
For $0\leq k\leq n$ with $k\neq m$, we let $N_k: X_n\to X_n$ be as follows:
  \begin{equation}
    N_k (w) = \mu(w, w\circ d_k\circ s_{k'} , x\circ d_k \circ s_{k'})
  \end{equation}
    We can then define $w_{k} = N_k(w_{k-1})$  for $0\leq k <m$, 
    then we take 
    $w_{n+1}=w_{m-1}$ and define 
    $w_k = N_k(w_{k+1})$ for $m<k\leq n$.
And finally, we define 
  $
\text{lift}(x,y):= w_{m+1}
$.
\end{construction}

  In the proof of Theorem 3 of \cite{KanFibrationsForMalcevAlgebras} the following is shown:
\begin{theorem}\label{ThmLiftReallyLifts}
  Suppose that $X,Y$ are Malcev algebras and $\alpha$ has a degeneracy-section in Diagram \ref{eqn:dfnKanFib}.
  Then the map $\mathrm{lift}(x,y)$ from 
  Construction \ref{constructionMalcevLiftsHelpers}  
  solves the lifting problem. 
\end{theorem}
  Now suppose that there is a degenerate solution to our lifting problem, 
  meaning there is some $g:\Delta^n\to X$ such that our lifting problem 
  is of the following form:
\begin{equation}\begin{tikzcd}\label{eqnLiftProbWithDegSol}
        \Lambda^{n+1}_m \arrow[rr,"x"] \arrow[d,hook]&& X \arrow[d,"\alpha"] \\
        \Delta^{n+1} \arrow[r, "s_j"] & \Delta^n \arrow[r,"y"] \arrow[ru,"g"] & Y
\end{tikzcd}\end{equation}
We shall show that in this case 
$\text{lift}(x,y\circ s_j) = g\circ s_j $.

For the rest of this section, we will say 
that $k$ is \textbf{encountered before} $l$ or $l$ is 
\textbf{encountered after} $k$ is 
$w_k$ is defined before $w_l$ in construction \ref{constructionMalcevLiftsHelpers}.
With this terminology, we can formulate the rest of our proof strategy:
\begin{itemize}
  \item
  We will first prove by induction that whenever $k$ is encountered before $j,j+1$, 
    we have that $w_k$ is of the form $z\circ s_j$ for some $z\in X_{n}$. 
\item
  We will then prove that for $k$ the first among $j,j+1$ we encounter
    we have $w_k = g\circ s_j$. 
\item We will then show that $w_k=g\circ s_j$ for all $k$ encountered later.
\end{itemize}

\begin{lemma}
  For a lifting problem 
  with a degenerate solution 
  as in Equation \ref{eqnLiftProbWithDegSol}
  and
  $w_k$ as in Construction \ref{constructionMalcevLiftsHelpers}, 
  whenever $w_k$ is defined before $w_j$ and $w_{j+1}$, we have $w_k= z\circ s_j$ 
  for some $z\in X_n$.
\end{lemma}
\begin{proof}
  We will proceed by induction. 
We defined $w_{-1} = \beta_{n+1}(y\circ s_j)$, and because $\beta$ was a degeneracy-section, 
we get that $w_{-1} = \beta_n(y)\circ s_j$.

Now let $0\leq k <m$ or $m<k \leq n+1$, 
suppose $k$ is encountered before $j,j+1$, and 
whenever $l$ is encountered before $k$, we have 
  $w_l=z_l\circ s_j$ for some $z_l\in X_n$. 

If $k=n+1$, we have that $w_{k}=w_{m-1}$, which is of the form $z_{m-1}\circ s_j$ 
by the induction hypothesis. 
If $k\neq n+1$, we have $0\leq k \leq n$ and $k\neq m$, and
we defined 
$w_k = N_k(w_l)$ for some earlier defined
$w_l\in X_{n+1}$. By the induction hypothesis, $w_l=z\circ s_j$ for some $z\in X_n$. 
Now 
  \begin{equation}
    N_k(w)=N_k(z\circ s_j) = 
    \mu\big(
                z\circ s_j ,
                z\circ s_j \circ d_k \circ s_{k'} , 
                x\circ d_k\circ s_{k'}
                \big)
  \end{equation}
  Note that $x\circ d_k = g\circ s_j\circ d_k$ by Equation \ref{eqnLiftProbWithDegSol}. 
We will now calculate $s_j\circ d_k \circ s_{k'}$. 
As $k$ is encountered before $j,j+1$, we have $k\neq j,j+1$. 
We will now make a case distinction on whether 
$k<j$ or $k>j+1$.
        \begin{itemize}
                \item If $k>j+1$, we have that 
                        $s_j\circ d_k = d_{k-1}\circ s_j$.
                        Also it follows that $j<k'$, and therefore 
                        $s_j\circ s_{k'} = s_{k'-1} \circ s_j$. 
                        In this case, we may conclude that 
                        $
                                s_j\circ d_k \circ s_{k'} = 
                                d_{k-1}\circ s_{k'-1} \circ s_j.
                        $
                \item If $k<j$, we have that 
                        $s_j\circ d_k = d_k \circ s_{j-1}$.
                        Also it follows that $k'<j$, and therefore
                        $s_{j-1}\circ s_{k'} = s_{k'} \circ s_{j}$. 
                        Hence 
                        $
                                s_j\circ d_k \circ s_{k'} = 
                                d_{k}\circ s_{k'} \circ s_j.
                        $
        \end{itemize}
        Hence for $k\neq j,j+1$, there is some morphism $f$ in $\Delta$ such that 
        $
                s_j\circ d_k \circ s_{k'} = 
                f \circ s_j.
        $ 
        For this $f$, it follows that 
        \begin{equation}
           N_k (z\circ s_j) = 
                \mu\big(
                z\circ s_j ,
                z\circ s_j \circ d_k \circ s_{k'} , 
                g\circ s_j \circ d_k\circ s_{k'}
                \big) = 
                \mu\big(
                z\circ s_j ,
                z\circ f\circ s_j ,
                g\circ f\circ s_j 
                \big) 
        \end{equation}
        Now as $X$ is a simplicial Malcev algebra, $\mu$ respects the 
        simplicial morphisms and we have 
        \begin{equation}
          w_k =N_k(w)=N_k(z\circ s_j) = 
                \mu\big(
                z\circ s_j ,
                z\circ f\circ s_j ,
                g\circ f\circ s_j 
                \big) 
                = 
                \mu\big(
                z,
                z\circ f,
                g\circ f
                \big) \circ s_j
        \end{equation}
        By induction, it follows that for all $k$ encountered before $j,j+1$, 
        there is some $z\in X^n$ such that 
        $w_k=z\circ s_j$. 
\end{proof}
        Now we will consider what happens when we construct $w_j$ or $w_{j+1}$.
\begin{lemma}
  For $k$ the first of $j,j+1$ we encounter, we have 
 $w_k = g\circ s_j$. 
\end{lemma}
\begin{proof}
  Let $k$ be the first of $j,j+1$ we encounter. 
  Whether we first encounter $j$ or $j+1$ depends on whether $j<m$ or $j>m$: 
  \begin{itemize}
    \item Note that if $j>m$, we have $k=j+1$. As $k>m$, we have $k'=k-1=j$. 
    \item  If $j<m$, we have $k=j$. As $k<m$, we have $k'=k=j$. 
  \end{itemize}
  So in both cases, we have $k' =j$. Also 
  remark that $s_j\circ d_j = s_j\circ d_{j+1} = id$. Thus $s_j\circ d_k = id$. 
  By the previous Lemma, $w_k = N_k(z\circ s_j)$ for some $z\in X_n$. We get that
  \begin{equation}
    N_k(z\circ s_j) = 
  \mu\big(
    z\circ s_j , z\circ s_j\circ d_k \circ s_{k'} , g\circ s_j \circ d_k \circ s_{k'}
    \big) =
          \mu\big(z\circ s_j , z\circ s_{j} , g\circ s_j\big).
  \end{equation}
  And as the first two entries are equal, we get $w_k = g\circ s_j$, as required. 
\end{proof}

Now to see that for all $k$ encountered after $j,j+1$ also have $w_k=g\circ s_j$, 
consider the following lemma:
\begin{lemma}
  For any $k$, a solution to the lifting problem in Diagram \ref{eqn:dfnKanFib}
  is a fixed point for $N_k$. 
\end{lemma}
\begin{proof}
  For any solution $h$, we have 
  \begin{equation} 
    N_k(h) = 
   \mu\big(h, h\circ d_k \circ s_{k'} , x \circ d_k \circ s_{k'}\big)
  \end{equation}
  As $h$ is a solution we have $h\circ d_k = x\circ d_k$, thus the last two entries 
  in $\mu$ above are equal. Thus $N_k(h)=h$. 
\end{proof}
Therefore, we will have that $w_k=N_k(g\circ s_j)$ for all $k$ encountered after $j,j+1$. 
In particular, we will have that $\text{fil}(x,y\circ s_j) = g\circ s_j$ as required.

  So to summarize, to show Theorem \ref{thmMainMalcevThm}, we use Construction 
  \ref{constructionMalcevLiftsHelpers} to give a lifting structure for $\alpha$. 
  It is really a lifting structure by Theorem \ref{ThmLiftReallyLifts}.
  We showed that if a degenerate solution of the form $g\circ s_j$ exists, 
  the helper functions $w_k$ from Construction \ref{constructionMalcevLiftsHelpers} 
  will always be degenerate along $j$, and as soon as $k$ is encountered within or after 
  $\{j,j+1\}$, we have $w_k = g\circ s_j$. 
  In particular $\text{lift}(x,y\circ s_j) = g\circ s_j$. 
  Therefore the lifting structure for $\alpha$ is degenerate-preferring. 
\subsection{Some consequences}
In this section we give some corollaries of Theorem \ref{thmMainMalcevThm} 
relating the notions in Section \ref{sec:Definitions} back to the classical notions of 
Kan fibrations. 

\begin{corollary}\label{corSimpObjMalcevDegPrefKanCom}
  All inhabited simplicial algebras of a Malcev theory are degenerate-preferring Kan complexes.
\end{corollary}
\begin{proof}
  Let $X$ be a inhabited simplicial algebra of a Malcev theory. Then
  $X$ is inhabited iff $X_0$ is inhabited. 
  Let $x\in X_0$ witness this. Then the morphism $\beta:\Delta^0\to X$ given by 
  $\beta_n(*) = (s_0)^n(x)$ is a section, and hence a degeneracy-section of the 
  unique morphism $X\to \Delta^0$. 

  Note also that the terminal simplicial set $\Delta^0$ 
  can be given the structure of a Malcev algebra with operation $\mu(*,*,*)=*$.
  This structure is such that for any Malcev algebra $Y$ the morphism $Y\to \Delta^0$ 
  is a morphism of Malcev algebras. 

  Thus $X\to \Delta^0$ is a morphism of Malcev algebras with a degeneracy-section. 
  Therefore it is a degenerate-preferring Kan fibration. 
  Thus $X$ is a degenerate-preferring Kan complex. 
\end{proof}
Note that $X$ is inhabited as soon as we have a map $\Lambda^n_m \to X$. 
Therefore, the unique map $X\to \Delta^0$ can be assigned 
a degenerate-preferring lifting function against horn inclusions. 
Thus all simplicial algebras of a Malcev theory are degenerate-preferring Kan complexes. 
Under this consideration, we have the following corollary:
\begin{corollary}\label{corAllSimpKanMalcev}
  Let $\mathcal T$ be an algebraic theory. The following are equivalent:
  \begin{enumerate}[(i)]
    \item $\mathcal T$ is a Malcev theory. 
    \item All simplicial algebras of $\mathcal T$ are degenerate-preferring Kan complexes.
    \item All simplicial algebras of $\mathcal T$ are Kan complexes.
    \item All simplicial algebras of $\mathcal T$ are uniform Kan complexes.
    \item All simplicial algebras of $\mathcal T$ are symmetric effective Kan complexes.
    \item All simplicial algebras of $\mathcal T$ are effective Kan complexes.
  \end{enumerate}
\end{corollary}
\begin{proof}
$(i)\implies (ii)$ by the above. 
$(ii)\implies (iii),(iv),(v),(vi)$ by the discussion in Section \ref{sec:Definitions}. 
Any of $(iv),(v),(vi)$ imply $(iii)$, and 
$(iii)\implies (i)$ is Proposition \ref{thm:AllSimpObjectsKanImpliesMalcevTheory}.
\end{proof}
We mention the following as it was the starting question for the research in this paper:
\begin{corollary}
  Simplicial groups are degenerate-preferring Kan complexes. 
\end{corollary}
\begin{proof}
  This follows from Corollary \ref{corSimpObjMalcevDegPrefKanCom} and
  Example \ref{exampleGroupIsMalcev}
\end{proof}
\begin{remark}\label{rmkMalcevCat}
  Corollary \ref{corAllSimpKanMalcev} says that we have found all algebraic theories
  $\mathcal T$ such that simplicial algebras of $\mathcal T$ are 
  degenerate-preferring Kan. 
  However, we could look further than simplicial algebras. 
  There is also a notion of Malcev categories, 
  which are regular categories 
  with a condition on the lattice of relations. 
  In \cite{Malcev}, Malcev\footnote{
    This reference is in Russian, and the authors were 
        unable to find a translation. 
        The name Malcev is a transcription, and has also been transcribed
        as Mal'cev, Maltsev and Mal'tsev.}
  showed that for algebras, this is equivalent to 
  having a Malcev operation. 
  Let us mention Theorem 4.6 of \cite{CKP}
  which states that a regular category $\mathcal A$ is a Malcev category iff every simplicial object of 
  $\mathcal A$ is a Kan complex.
\end{remark}

%% file: LiftingAWFS.tex
The classical Kan fibrations are part of the ``right class" in the weak factorisation system 
of the Kan-Quillen model structure. 
In this section, we will discuss a constructive variation on weak factorisation systems, 
namely that of an algebraic weak factorisation system (awfs). 
We will show that the symmetric effective and the effective Kan fibrations both fit into 
their own awfs. 
In the first two subsections, we will recall the definition of a double category and an awfs; the definition of an awfs we give here is the one from \cite{Bourke}, where it is called a lifting algebraic weak factorisation system. Indeed, 
these sections are based on \cite{Bourke} and contain no original work. 
We will then define a class of double categories based on horn inclusions and 
study awfss where the left double category comes from this class. 
We will then characterize the vertical morphisms of the 
right double category of such an awfs and show that the 
effective and symmetric effective Kan fibrations satisfy this characterization. 
\subsection{Double categories}
  A (small) \textbf{double category} consists of four sets, 
  called objects, horizontal arrows, vertical arrows and squares. 
These carry the following structure: 
  \begin{itemize}
    \item Both sets of arrows have domains and codomains in the objects, 
      an identity arrow for each object, 
      and an associative composition structure. 
    \item The squares have horizontal and vertical domains and codomains. 
      For subscripts $h,v$ horizontal and vertical and $d,c$ the domain and codomain, 
      these match up as follows:
      \begin{equation}\begin{tikzcd}
        A\arrow[r,"d_h"] \arrow[d,"d_v"']& B \arrow[d,"c_v"] \\
        C \arrow[r, "c_h"'] & D
      \end{tikzcd}\end{equation}
\item For both the horizontal and vertical arrows, there are identity squares and there 
  there is a horizontal and a vertical composition of squares. 
      The horizontal composition interacts with vertical domain and codomain as one would expect, and vice versa. 
      Also, if four squares match up so they can be composed in any order, the order of composition does not matter. 
  \end{itemize}

\begin{example}
  For any category $\mathcal C$, we have the category $\mathbb S\textbf{q}(\mathcal C)$, 
  whose objects are objects from $\mathcal C$, both vertical and horizontal arrows are arrows from $\mathcal C$, 
  and squares are commutative squares from $\mathcal C$. 
\end{example}
We will denote $\alpha:l\to r$ 
for $\alpha$ a square with vertical domain $l$ and codomain $r$. 
If $\alpha$ is a square of $\mathbb S\textbf{q}(\mathcal C)$ with horizontal domain $u$ and codomain $v$, 
we will write $\alpha=[u,v]$.

A \textbf{double functor} is the natural notion of morphism between double categories. 
 
If we have a double functor 
$A : \mathbb A\Rightarrow\mathbb S\textbf{q}(\mathcal C)$, 
we will call $\mathbb A$ a \textbf{double category over $\mathcal C$}.

If we have a double functor $S: \mathbb A \Rightarrow \mathbb B$ which is the identity 
on the objects and both classes of morphisms, and is injective on the squares, 
we call $\mathbb A $ a \textbf{wide subcategory} of $\mathbb B$.

\subsection{Algebraic weak factorisation systems}\label{sec:LiftingAWFS}
Let $
  L:\mathbb L \Rightarrow \mathbb S\textbf q (\mathcal C), 
  R:\mathbb R \Rightarrow \mathbb S\textbf q (\mathcal C)
$.  A \textbf{lifting operation} between $\mathbb L$ and $\mathbb R$ 
is a family $\phi$ which
assigns for all vertical morphisms $l:A\to B$ of $\mathbb L$ and 
  $r:X\to Y$ of $\mathbb R$, 
  and each square $[u,v]:Ll\to Rr$ of $\mathbb S\textbf q (\mathcal C)$, 
a choice of diagonal filler $\phi_{l,r}(u,v):LB\to RX$. 
This choice must satisfy the following conditions:
\begin{itemize}
        \item We have horizontal compatibilities, which is about 
                horizontal composition of squares.
                \begin{equation}
                  \text{ For any }
                        \begin{tikzcd}
                          A' \arrow[d,"l'"'{name=Ll}]
                                \arrow[r,"a"] 
                                & A \arrow[d, "l"{name=Lr}] \\
                                B' \arrow[r,"b"']& B
                                \arrow[from=Ll, to=Lr, phantom ,"\lambda"]
                        \end{tikzcd} \text{ square in $\mathbb L$ and }
                        \begin{tikzcd}
                          X \arrow[d,"r"'{name = Rl}]
                                \arrow[r, "x"]
                                & X' \arrow[d,"r'"{name=Rr}] \\
                                Y 
                                \arrow[r,"y"']
                                & Y' 
                                \arrow[from=Rl, to=Rr, phantom ,"\rho"]
                        \end{tikzcd}
                        \text{ square in $\mathbb R$,}
                \end{equation}
                and any diagram in $\mathcal C$ of the form:
                \begin{equation}\begin{tikzcd}
                  LA' \arrow[d,"Ll'"'{name = Ll}]
                                \arrow[r,"La"] 
                                & LA \arrow[d, "Ll"{name = Lr}] 
                                \arrow[r,"u"]
                                &
                                RX \arrow[d,"Rr"'{name=Rl}]
                                \arrow[r, "Rx"]
                                & RX' \arrow[d,"Rr'"{name=Rr}] 
                                \\
                                LB' \arrow[r,"Lb"']
                                & LB
                                \arrow[r,"v"']
                                &
                                RY 
                                \arrow[r,"Ry"']
                                & RY' 
                                \arrow[from=Rl, to=Rr, phantom ,"R\rho"]
                                \arrow[from=Ll, to=Lr, phantom ,"L\lambda"]
                \end{tikzcd}\end{equation}
                we require that the canonical choices of lifts $LB'\to RX'$ are equal. 
                In particular:
                \begin{itemize}
                \item If $\rho$ is a vertical  identity square in
                        $\mathbb R$, we get the 
                                \textbf{left horizontal compatibility condition}:
                \begin{equation}\label{eqnLeftHorizontalCompatibility}
                        \phi_{l',r}(u\circ La, v\circ Lb) =
                        \phi_{l,r}(u,v)\circ Lb. 
                \end{equation}
                \item If $\lambda$ is a vertical identity square in
                        $\mathbb L$, we get the 
                                \textbf{right horizontal compatibility condition}:
                \begin{equation}\label{eqnRightHorizontalCompatibility}
                        \phi_{l,r'}(Rx \circ u , Ry\circ v ) =
                        Ry \circ \phi_{l,r}(u,v)
                \end{equation}
                \end{itemize}
        \item  We have vertical compatibilities, which is about
                composition of vertical morphisms.
                For any composable vertical morphisms
                  $
                        \begin{tikzcd}[sep=small]
                          A\arrow[r,"l"] &
                                B \arrow[r,"l'"]&
                                C
                        \end{tikzcd} \text{ 
                        in $\mathbb L$ and }
                        \begin{tikzcd}[sep = small]
                          X\arrow[r,"r"] &
                                Y \arrow[r,"r'"]&
                                Z
                        \end{tikzcd} \text{ 
                        in $\mathbb R$,}
                        $
                and any diagram in $\mathcal C$ of the following form:
                \begin{equation}\begin{tikzcd}
                        LA \arrow[d,"Ll"]  \arrow[r,"u"]& RX \arrow[d,"Rr"] \\
                        LB \arrow[d,"Ll'"] & RY \arrow[d,"Rr'"] \\
                        LC \arrow[r,"v"]& RZ
                \end{tikzcd}\end{equation}
                we require that the canonical choices of lifts $LB'\to RX'$ are equal. 
                In particular: 
                        \begin{itemize}
                                \item If $l'$ is an identity, 
                                  we get the \textbf{right vertical compatibility condition}:
                                \begin{equation}\label{eqnRightVerticalCompatibility}
                                        \phi_{l,r'\circ r}(u , v ) =
                                        \phi_{l,r}\big(u, \phi_{l,r'}(Rr \circ u)\big)
                                \end{equation}
                                \item If $r'$ is an identity, 
                                  we get the \textbf{left vertical compatibility condition}:
                                \begin{equation}\label{eqnLeftVerticalCompatibility}
                                        \phi_{l'\circ l,r}(u , v ) =
                                \phi_{l',r}\big(\phi_{l,r}(u,v\circ l') , v\big)
                                \end{equation}
                        \end{itemize}
        \end{itemize}
        If $\phi$ is an $(\mathbb L, \mathbb R)$-lifting operation, we call 
        $(\mathbb L, \phi, \mathbb R)$ a \gls{lifting structure} on $\mathcal C$.

\begin{definition}\label{dfnRLP}
For $\mathbb L$ a double category over $\mathcal C$, 
we define a double category $\RLP(\mathbb L)$ over $\mathcal C$:
\begin{itemize}
  \item The objects and horizontal maps are objects and morphisms from $\mathcal C$. 
  \item Vertical maps are tuples $(f,\psi)$, where $f$ is a morphism from $\mathcal C$
and $\psi$ assigns to each 
    square $[u,v]:Ll\to f$ of $\Sq(\mathcal C)$
a lift $\psi_l(u{,}v):B\to X$, such that $\psi$ satisfies the left horizontal and vertical compatibility conditons from 
Equations (\ref{eqnLeftHorizontalCompatibility}) and (\ref{eqnLeftVerticalCompatibility}). 
\item 
  If $(f,\psi)$ and $(g,\chi)$ are vertical morphisms and 
    we have a commuting square $[x,y]:f\to g$ in $\mathbb S\textbf q(\mathcal C)$, we let 
    $[x,y]$ be a square in \textbf{RLP}($\mathbb L$) if 
      it respects the lifting structures, so 
    for any square $[u,v]:Ll\to f$, we have 
    \begin{equation}\chi_l(x\circ u,v\circ y) = x \circ \psi_l(u,v).\end{equation}
\end{itemize}
Vertical identity morphisms are given by $(1,t)$, where $t_l(u,v)=v$. 
Composition of vertical identity morphisms is given by 
\begin{equation}
  (f,\psi)\circ (g,\chi) = (f\circ g,\phi)
  \text{ with }
\phi_{l}(u , v ) =
\psi_{l}\big(u, \chi_{l}(f  \circ u)\big)
. \end{equation}
All other identities and compositions are as in $\mathcal C$. 
The map $\textbf{RLP}(\mathbb L)\to \mathbb S\textbf q (\mathcal C)$ sends vertical morphisms $(f,\psi)$ to $f$, 
and the objects, horizontal morphisms and squares to themselves. 
\end{definition}
Dually, given a double category $\mathbb R$ over $\mathcal C$, 
we have a double category $\textbf{\gls{LLP}}(\mathbb R)$ over $\mathcal C$. 

Every lifting structure $(\mathbb L, \phi, \mathbb R)$  on $\mathcal C$ induces double functors 
\[ \mathbb{R} \to \textbf{RLP}(\mathbb L) \quad \mbox{and} \quad \mathbb L \to \textbf{LLP}(\mathbb R) \]
sending vertical morphisms $f$ to $(f,\phi_{-,f})$ and $(f,\phi_{f,-})$, respectively.
  The lifting structure $(\mathbb L, \phi, \mathbb R)$  on $\mathcal C$ is called an 
  \textbf{algebraic weak factorisation system}, or
  \textbf{awfs}, if it satisfies the following two axioms:
  \begin{enumerate}
    \item\textbf{Axiom of lifting}: The double functors 
      $\mathbb R\hookrightarrow\textbf{RLP}(\mathbb L)$ and 
      $\mathbb L\hookrightarrow\textbf{LLP}(\mathbb R)$ are invertible. 
    \item \textbf{Axiom of factorisation}: each morphism $f:A\to B$ of $\mathcal C$ admits 
      a factorisation \begin{equation}\begin{tikzcd}
      A\arrow[r, "L(l_f)"] & C_f\arrow[r,"R(r_f)"] & B\end{tikzcd} \end{equation} 
      with $l_f,r_f$ vertical morphisms from $\mathbb L, \mathbb R$.
      This factorisation is bi-universal, meaning that 
  \begin{itemize}
    \item For any vertical morphism $l$ of $\mathbb L$ and square $\alpha:Ll\to f$ in $\mathbb S\textbf q (\mathcal C)$ 
      there is a unique square $\lambda:l\to l_f$ in $\mathbb L$ such that $\alpha$ factors as 
      $[1_A,R(r_f)]\circ L \lambda$.
\item For any  vertical morphism $r$ of $\mathbb R$ and square $\beta: f\to Rr$ in $\mathbb S\textbf q (\mathcal C)$ 
  there is a unique square $\rho:r_f\to r$ in $\mathbb R$ such that $\beta$ factors as
      $R\rho\circ [L(l_f),1_B]$.
  \end{itemize}
 \end{enumerate}
Now we will state without proof Proposition 18 from \cite{Bourke}:
\begin{proposition}\label{Prop:BourkeCofGenerated}
        Let $\mathcal C$ be a locally presentable category and let $\mathbb L$ be 
        a small double category over $\Sq(\mathcal C)$. 
        Then the canonical lifting structure 
        $(\LLP(\RLP(\mathbb L)), \phi, \RLP(\mathbb L))$ is an awfs.
\end{proposition}
Such lifting structures are called \textbf{cofibrantly generated}. 
It should be noted that the proof for the above proposition uses
a version of Quillen's small objects argument from \cite{GarnerSmallObjects}, 
where the axiom of choice was used. 
At the moment, it is not known whether a constructive proof of this result exists.

In the remainder of this section, we shall show that the effective and symmetric effective Kan fibrations
are vertical morphisms in cofibrantly generated lifting structures. 
\subsection{Degenerate-preferring Kan fibrations}\label{secDegPrefLawfs}
In this section, we will construct a subcategory of $\mathbb S \textbf q (\textbf{sSet})$
which cofibrantly generates an awfs where the vertical morphisms of the right double category are
degenerate-preferring Kan fibrations. 

We define the double category $\mathbb D$ as follows:
for our objects, we take the set of simplices $\Delta^n$ and horns $\Lambda^n_m$. 
For our vertical morphisms, we take identities on the objects and horn inclusions. 
The horizontal morphisms $\Delta^a \to \Delta^{b}$ are epimorphisms, 
the horizontal morphisms $\Lambda^n_m \to \Lambda^{n'}_{m'}$ are identities and 
the horizontal morphisms $\Lambda^n_m \to \Delta^{b}$ are of the form $s\circ \iota$, where
$\iota$ is a horn inclusion 
and $s: \Delta^n\to \Delta^{b}$ is a non-identity epimorphism. 
For $s,\iota$ as above, 
the squares are either identity squares on vertical or horizontal arrows or of the form:
\begin{equation}\label{eqnDsquares}\begin{tikzcd}
  \Lambda^n_m \arrow[r, "s\circ \iota"] \arrow[d,hook,"\iota"]& \Delta^{b} \arrow[d, "id"]\\
  \Delta^n \arrow[r, "s"] & \Delta^{b}
\end{tikzcd}\end{equation}
Note that these sets of arrows and squares are closed under any composition.

\begin{theorem}
  The vertical arrows of $\RLP(\mathbb D)$ are exactly the degenerate-preferring Kan fibrations. 
\end{theorem}
\begin{proof}
As having lifts against identities is trivial, having a choice of lifts against the vertical morphisms of $\mathbb D$ 
is equivalent to being a functional Kan fibration.

Now suppose that $(\alpha,\text{lift})$ is a degenerate-preferring Kan fibration. 
We must show that it satisfies the horizontal and vertical compatibility conditions against squares in $\mathbb D$. 

The left vertical compatibility conditions are trivially satisfied as 
whenever two vertical arrows can be composed, one of them is an identity. 
Satisfying the left horizontal compatibility conditions against identity squares is also trivial,
thus the only squares for which it is not trivial to satisfy the left horizontal compatibility condition are those 
as in Diagram \ref{eqn:dfnKanFib}. 

Suppose the following diagram commutes for $\iota$ a horn inclusion and $s$ a non-identity degeneracy map:
\begin{equation}\begin{tikzcd}
  \Lambda^n_m \arrow[r, "s\circ \iota"] 
  \arrow[d,hook,"\iota"]& 
  \Delta^{b} \arrow[d, "id"] \arrow[r,"z"] & X \arrow[d,"\alpha"]\\
  \Delta^n \arrow[r, "s"] 
  & \Delta^{b} \arrow[r,"f"] & Y
\end{tikzcd}\end{equation}
Then $s=s'\circ s_j$ for some degeneracy map $s_j$, hence $z\circ s$ is a degenerate solution for the lifting 
problem against $\iota$ and therefore $\text{lift}(z\circ s\circ \iota, f\circ s) = z\circ s$ as $\alpha$ is degenerate-preferring. Since 
$z$ must be the unique lift against the identity arrow, $(\alpha,\text{lift})$ satisfies 
the left horizontal compatibility condition against the above squares as in Diagram \ref{eqn:dfnKanFib} 
and is a vertical morphism of $\RLP(\mathbb D)$. 

Conversely, suppose that a degenerate solution $z\circ s_j$ to a lifting problem as in Diagram \ref{eqn:dfnKanFib} exists, 
then the square in Diagram \ref{eqn:dfnKanFib} factors as the above diagram for $s=s_j$. 
Note that the right square has $z$ as unique lift. 
  By the 
  horizontal compatibility condition, any 
  vertical morphism $(f,\psi)$ of $\RLP(\mathbb D)$ therefore assigns $z\circ s_j$ 
  as solution to Diagram \ref{eqn:dfnKanFib}. 
  Therefore any such a vertical morphism is a degenerate-preferring Kan fibration. 

We conclude that the vertical morphisms of $\RLP(\mathbb D)$ are exactly the degenerate-preferring Kan fibrations. 
\end{proof}
\subsection{The double category of horn pushout sequences}
In the following sections, we will prove some general lemmas in order to construct an awfs based on horn inclusions. 
To stay as general as possible, we will assume we are given a surjection 
$\ell:\mathcal L \to \Lambda$ onto the set of horn inclusions; in the cases of interest, we have $\mathcal L = \Lambda$ and $\ell = 1_\Lambda$, or $\mathcal L$ the set of signed horn inclusions and $\ell = \pm$ the forgetful map which forgets the sign.
In this section, we will build a double category $\mathbb L_\ell$ over $\mathbb S\textbf q(\textbf{sSet})$. 
We will later study right vertical maps of an awfs cofibrantly generated by 
wide subcategories of $\mathbb L_\ell$. 

The objects of $\mathbb L_\ell$ will be decidable sieves: 
\begin{definition}
  A \textbf{decidable sieve} of $\Delta^n$ is a sieve $S\subseteq \Delta^n$ 
  such that for any $0\leq m\leq n$ and 
  any $p:\Delta^m\to \Delta^n$, it is decidable whether or 
  not $p$ factors through $S$.
\end{definition}

The horizontal morphisms in $\mathbb L_\ell$ between 
decidable sieves $S\subseteq \Delta^a, T\subseteq \Delta^b$ are morphisms
$f:[a]\to[b]$ in $\Delta$ such that $f^*(T) = S$; in other words, they are maps $f:[a] \to [b]$ fitting into pullback squares
\begin{displaymath}
  \begin{tikzcd}
    S \ar[r] \ar[d, hook] & T \ar[d, hook] \\
    {[a]} \ar[r, "f"] & {[b]}
  \end{tikzcd}
\end{displaymath} 
Horizontal compositions and identities are then taken in $\Delta$. 

The vertical morphisms will be horn pushout sequences:
\begin{definition}\label{dfn:pushoutsequence}
  A \gls{horn pushout sequence} from $S_0\subseteq \Delta^a$ to $S_k\subseteq \Delta^a$ 
  is a finite sequence of composable monomorphisms of the form
        \begin{equation}\label{eqn:Sseq}\begin{tikzcd}
        S_0    \arrow[r,hook, "\sigma_1"] &
        S_1    \arrow[r,hook, "\sigma_2"] & 
        \cdots \arrow[r,hook, "\sigma_k"] &
        S_k\arrow[r,hook] & \Delta^a.
\end{tikzcd}\end{equation} 
with length $k\geq 0$, 
with for every $\sigma_i$ a chosen $\iota_i\in\mathcal L$ 
  and a square 
        \begin{equation}\label{eqn:chosenpushout}\begin{tikzcd}
          \Lambda^n_m \arrow[r,hook]\arrow[d,hook, "\ell(\iota_i)"'] 
\arrow[rd, draw=none, "\ulcorner"{pos=0.875,anchor=center}] 
\arrow[rd, draw=none, "\lrcorner"{pos=0.125,anchor=center}] 
&
S_{i-1} \arrow[d,hook,"\sigma_i"] \\
\Delta^n \arrow[r,hook] & S_i
\end{tikzcd}\end{equation}
which is both a pushout and a pullback.
Such sequences can be composed and the identity is given by the empty sequence. 
\end{definition}
\begin{remark}\label{rmkCount}
We remark that for each $a\in\mathbb N$, 
there are only a finite number of monomorphisms into  $[a]$. 
Each of these correspond to a non-degenerate simplex of $\Delta^a$. 
The non-degenerate simplices that factor through a decidable sieve $S\subseteq \Delta^a$ 
can be counted. For a mono $f: S\hookrightarrow T$, we say that $f$ adds a non-degenerate simplex if 
it factors through $T$ but not through $S$.

  If we have a pullback square $\alpha: f\to g$ in $\textbf{sSet}$ where all maps are monos, 
  it corresponds to an intersection, therefore any non-degenerate simplex added by $f$ is also added by $g$. 
  Thus $g$ adds more non-degenerate simplicices than $f$. 
  If $\alpha$ is in addition also a pushout square, it corresponds to a union, and $f$ and $g$ add
  exactly the same amount of non-degenerate simplices. 

  In particular, in 
  Diagram \ref{eqn:chosenpushout},
the sieve $S_{i}$ contains exactly two more non-degenerate simplices than $S_{i-1}$, namely, one non-degenerate $n$-simplex and one non-degenerate $n-1$-simplex.
  As a consequence, in 
  Diagram \ref{eqn:Sseq}, the sieve 
  $S_k$ has exactly $2k$ more non-degenerate simplices than $S_0$. 
\end{remark}
For the squares, our first attempt would be the following definition:
\begin{definition}\label{dfn:hornpushoutPullback}
        A \textbf{pullback square of horn pushout sequences} from a horn pushout 
        sequence as in Definition \ref{dfn:pushoutsequence} into a horn 
        pushout sequence of the form 
        \begin{equation}\label{eqn:Tseq}\begin{tikzcd}
                T_0    \arrow[r,hook, "\tau_1"] &
                T_1    \arrow[r,hook, "\tau_2"] & 
                \cdots \arrow[r,hook, "\tau_l"] &
                T_l\arrow[r,hook] & \Delta^b,
        \end{tikzcd}\end{equation} 
        consists of a morphism $f:[a]\to[b]$ in $\Delta$ 
        and a 
        nondecreasing function $\mu:\mathbb N_{\leq l}\to \mathbb N_{\leq k}$ 
        with $\mu(0)=0$ and $\mu(l)=k$ such that for each 
        $0\leq i \leq l$, we have that $f^*(T_i) = S_{\mu(i)}$.

For such squares, both horizontal domain and codomain correspond to $f$. 
Horizontal composition works by composing the sequences and the nondecreasing functions. 
Horizontal identities have vertical domain and codomain the empty horn pushout sequence and 
identity function on $\mathbb N_{\leq 0}$ as nondecreasing function. 

A pullback square from $S_0 \hookrightarrow S_k\hookrightarrow \Delta^a$ to $T_0\hookrightarrow T_l\hookrightarrow \Delta^b$ 
can be vertically composed with a square
$S'_0 \hookrightarrow S'_{k'}\hookrightarrow \Delta^{a'}$ to $T'_0\hookrightarrow T_{l'}\hookrightarrow \Delta^{b'}$ 
iff their underlying morphism $f$ matches and $S_k= S'_0, T_l=T'_0$, 
so for $\mu,\nu$ their respective nondecreasing functions
\begin{equation}(\mu+\nu) (i) = \begin{cases}
  \mu(i) \text { if } 0\leq i\leq l\\
  k + \nu(i-l) \text { if } l\leq i \leq l'+l
\end{cases} 
\end{equation}
is uniquely defined on $l$ and defines 
a function $\mu+\nu:\mathbb N_{\leq l+l'}\to \mathbb N_{\leq k+k'}$
satisfying the required properties to give a pullback square. 
This is how we define the vertical composition of pullback squares. 
The vertical identity is given by taking the identity for both $f$ and $\mu$. 
\end{definition}

However, such squares do not give us enough information for our purposes. 
We need something more fine-grained. We will explicitly decompose pullback squares as above
into smaller squares. 
\begin{definition}\label{dfn:hornpushoutMorphism}
  If the underlying morphism of a pullback square of horn pushout sequences is a face or degeneracy map, 
  we call the pullback square a \textbf{face square} or \textbf{degeneracy square} respectively. 
  A \textbf{square of horn pushout sequences} is a sequence of horizontally composable face or degeneracy squares. 
\end{definition}
The squares in $\mathbb L_\ell$ will be squares of horn pushout sequences. 
The domains and codomains of such squares are those of the pullback square of 
horn pushout sequences we get after horizontal composition. 
Both horizontal and vertical identities are empty sequences, 
and both horizontal and vertical compositions go by composing sequences, 
which satisfies all required associativity conditions. 

Finally, we define a forgetful double functor $L_\ell:\mathbb L_\ell \to \mathbb S\textbf q(\textbf{sSet})$. 
This means that 
\begin{itemize}
  \item The sieve $S\subseteq \Delta^a$ is sent to the simplicial set $S$. 
  \item A horizontal morphism corresponding to the pullback
          $f^*(T)=S$ is sent to the map $S\to T$. 
  \item A horn pushout sequence as in Definition \ref{dfn:pushoutsequence}
  is sent to the composition $S_0\hookrightarrow S_k$. 
\item A square of horn pushout sequences as in Definition \ref{dfn:hornpushoutMorphism} is 
    first horizontally composed to a pullback square of horn pushout sequences as in 
  Definition \ref{dfn:hornpushoutPullback} and then sent to the square
\begin{equation}\begin{tikzcd}
        S_0 \arrow[r]   \arrow[d,hook] 
        \arrow[dr, draw = none, "\lrcorner"{anchor = center, pos = 0.125}]
        &
        T_0  \arrow[d,hook]
        \\ 
        S_k \arrow[r] & T_l 
\end{tikzcd}\end{equation}
\end{itemize}

Now note that sieves in simplicial sets form a set, $\mathcal L$ is a set, and therefore $\mathbb L_\ell$ is small. 
Also as $\Delta$ is small, the presheaf category $\textbf{sSet}$ is locally presentable. 
We can thus apply Proposition \ref{Prop:BourkeCofGenerated}
on any wide subcategory of $\mathbb L_\ell$.

\subsection{Functional lifts respect face squares}
In this section, we will study $\RLP(\mathbb S)$ 
for specific wide subcategories $\mathbb S$ of $\mathbb L_\ell$. 
We say that a vertical morphism $(f,\psi)$ of $\RLP(\mathbb S)$ respects a square if it satisfies the left horizontal 
compatibility condition against that square. 
We will show that functional (signed) Kan fibrations correspond with vertical morphisms that respect face squares.

\begin{remark}\label{rmkUniqueOnLength1Seq}
  Note first that the left vertical compatibility conditions means that 
vertical morphisms of $\RLP(\mathbb S)$ are uniquely determined by their 
action on horn pushout sequences of length $1$. 
\end{remark}

Now we say $\mathbb S$ is \textbf{vertical decomposition closed} if whenever a square $\alpha$ in $\mathbb S$ 
can be written as vertical composition of two squares $\beta,\gamma$ in $\mathbb L_\ell$, 
both $\beta, \gamma$ are also squares in $\mathbb S$. 

For such $\mathbb S$, the squares which we respect are uniquely determined by the 
squares whose vertical codomain is a pushout sequence of length $1$. 

\begin{lemma}\label{lemSquaresWhoseTargetisLength1}
  Let $\mathbb S, \mathbb S'$ denote two wide subcategories of $\mathbb L_\ell$. 
  Assume $\mathbb S$ is vertical decomposition closed. 
  Let $(f,\psi)$ be a vertical morphism of $\RLP(\mathbb S')$ and 
  suppose it respects all squares in $\mathbb S$ whose vertical codomain is a pushout sequence of length $1$. 
  Then $(f,\psi)$ respects all squares in $\mathbb S$.
\end{lemma}
\begin{proof}
  Let $(f,\psi)$ be a vertical morphism of $\RLP(\mathbb S')$ respecting squares 
  in $\mathbb S$ whose vertical codomain has length $1$. 
  Let $\lambda : \sigma \to \tau$ be a square in $\mathbb S$. 
  We will use induction on the length of $\tau$.
  If $\tau$ has length $0$, it is an identity, which is preserved under pullbacks, 
  and hence $\sigma$ must be an identity also.
  Therefore both have unique lifts. 

  Now suppose $(f,\psi)$ respects squares in $\mathbb S$ whose vertical codomain has length $\leq l$
  for $l\geq 1$. And let $\tau$ have length $l+1$. Then we can write 
  $\tau$ as vertical composition $\tau^1\circ \tau^l$ where $\tau^1$ has length $1$ and $\tau^l$ has length $l$. 
  As $\lambda$ is a pullback square of horn pushout sequences, 
  we can decompose it as the two left squares in the following diagram:
  \begin{equation}\begin{tikzcd}
	{S_0} & {T_0} & X \\
	{S_{\mu(l)}} & {T_l} \\
	{S_k} & {T_{l+1}} & Y
  \arrow[from=1-1, to=2-1, "\sigma^a"']
	\arrow[from=2-1, to=3-1, "\sigma^b"']
	\arrow["{\tau^l}", from=1-2, to=2-2]
	\arrow["{\tau^1}", from=2-2, to=3-2]
	\arrow[from=1-1, to=1-2]
	\arrow[from=3-1, to=3-2]
	\arrow[from=1-3, to=3-3,"f"]
	\arrow[from=3-2, to=3-3]
	\arrow[from=1-2, to=1-3]
	\arrow[from=2-1, to=2-2]
	\arrow["\lrcorner"{anchor=center, pos=0.125}, draw=none, from=2-1, to=3-2]
	\arrow["\lrcorner"{anchor=center, pos=0.125}, draw=none, from=1-1, to=2-2]
  \arrow[blue,dashed, from=2-2, to=1-3, bend right = 10]
	\arrow[red,dashed, from=2-1, to=1-3, bend right = 10]
	\arrow[blue , dotted, from=3-2, to=1-3, bend right = 20]
	\arrow[red, dotted, from=3-1, to=1-3, bend right = 20]
  \end{tikzcd}\end{equation}
  As $\mathbb S$ was vertical decomposition closed, both of the two left squares lie in $\mathbb S$. 
  By the induction hypothesis, both the triangles involving the dashed and dotted arrow commute. 
  By the left vertical compatibility condition, both the triangles
  involving the red and blue arrows commute. 
  Therefore the lifts $S_k \to X$ are equal, and $(f,\psi)$ respects $\lambda$. 
\end{proof}

\begin{remark}\label{lemFaceDegeneracySquaresAreEnough}
  Let $\mathbb S, \mathbb S'$ be two wide subcategories of $\mathbb L_\ell$, 
  with $\mathbb S$ vertical decomposition closed. 
  By the left horizontal compatibility conditions, 
  a vertical morphism $(f,\psi)$ of $\RLP(\mathbb S')$ 
  respects all squares in $\mathbb S$ as soon as 
  it respects all face and degeneracy squares in $\mathbb S$. 
\end{remark}
Note that for any $\iota\in \mathcal L$ with $\ell (\iota): \Lambda^n_m\hookrightarrow \Delta^n$, 
both $\Lambda^n_m$ and $\Delta^n$ are decidable sieves of $\Delta^n$ and therefore we have a horn pushout sequence of length $1$ 
corresponding to $[id,id]: \ell (\iota) \to \ell (\iota)$.
If $(f,\psi)$ is a vertical morphism in ${\bf RLP}(\mathbb S)$, we call its lift against 
this sequence the lift against $\iota$. 
\begin{lemma}\label{lemFaceSquares}
A vertical morphism $(f,\psi)$ of $\RLP(\mathbb S)$ 
  assigns to pushouts sequences of length $1$ with chosen $\iota\in\mathcal L$ 
  the pushout of the lift against $\iota$ iff 
  $(f,\psi)$ respects face squares.
\end{lemma}

\begin{proof}
\item 
\begin{itemize}
\item
  Suppose that $\psi$ works by pushing out the lift against $\iota$. 
  By Lemma \ref{lemSquaresWhoseTargetisLength1}, we need to show that 
  $(f,\psi)$ respects face squares into a sequence of length 1. 

  So let $\delta:\sigma \to \tau$ be a face square and let $\tau$ have length $1$ with chosen $\iota\in\mathcal L$. As $\delta$ is a face square, the square $L(\sigma)\to \tau_1$ is a pullback square in ${\bf sSet}$ in which all arrows are mono. Therefore, by Remark \ref{rmkCount},  $L(\sigma)$ adds at most 2 faces, thus $\sigma$ has either length 0 or 1. 
  If the length is $0$, there is a unique lift against $\sigma$ and we are done. 

  Suppose the length is $1$. As the pullback along a mono corresponds to taking intersection,
  we are in the situation that $\sigma_1$ and $\tau_1$ add the same two simplices, and the solid part of the 
  following diagram commutes:
  \begin{equation}\label{eqnNoChangeSquare}\begin{tikzcd}
    \Lambda^n_m \arrow[d,hook,"\ell(\iota)"] \arrow[r,hook] & 
    S_0 \arrow[d,"\sigma_1"'{name=s1},hook] \arrow[r,hook] & 
    T_0 \arrow[d,"\tau_1"{name=t1},hook] \arrow[r]& X\arrow[d,"f"]
    \\
    \Delta^n \arrow[r,hook]& 
    S_1 \arrow[r,hook] \arrow[rru,dashed,bend right= 10] & 
    T_1 \arrow[ru,dashed,bend right= 10]  \arrow[r]& Y
    \arrow[from=s1, to=t1, phantom, "L(\delta)"]
  \end{tikzcd}\end{equation}
  By assumption, now both the lift against $\tau_1$ and $\sigma_1$ are defined by pushing out the lift against $\iota$. 
  By the uniqueness property of the pushout, the two lifts $S_1\to X$ must match and we respect $\delta$. 
\item For the converse, assume that $(f,\psi)$ respects face squares. Consider a pushout sequence $\tau$ of length $1$ 
as in 
    the following diagram: 
    \begin{equation}\label{eqnPushoutSeqlength1}\begin{tikzcd}
  \Lambda^n_m \arrow[r,hook] \arrow[d,hook,"\ell(\iota)"']
  \arrow[rd, draw=none, "\ulcorner"{pos=0.875,anchor=center}] 
  \arrow[rd, draw=none, "\lrcorner"{pos=0.125,anchor=center}] 
  & T_0 \arrow[d, hook , "\tau_1"] \\ 
  \Delta^n \arrow[r,hook] 
      & T_1 \arrow[d,hook] \\
      & \Delta^b
\end{tikzcd}\end{equation}
The map $\Delta^n\to \Delta^b$ can be seen as composition of face maps
    by repeatedly taking faces which still contain the added simplices. 
Whenever we take such a face, we get a face square between pushout sequences of length 1, which 
we compose until we arrive at a horn pushout sequence with codomain $\Delta^n\hookrightarrow \Delta^n$. 
    Because $(f,\psi)$ respects face squares, 
    the chosen lift against $\tau$ composed with the map $\Delta^n \to T_1$ will be the lift against $\iota$. 
By the universal property of the pushout, 
the chosen lift against $\tau$ must therefore be the pushout of the lift against $\iota$.
\end{itemize}\end{proof}
Now let $1_\Lambda$ be the identity surjection on horn inclusions and 
    $\pm$ be the forgetful surjection from signed horn inclusions to horn inclusions. 
    Let $\mathbb F_\Lambda, \mathbb F_\pm$ be the wide subcategories of $\mathbb L_{1_\Lambda}, \mathbb L_\pm$ 
    respectively containing only compositions of face squares. 
  \begin{corollary}\label{corFunctionalLAWFS}
    The vertical morphisms of $\RLP(\mathbb F_{1_\Lambda})$ and $\RLP(\mathbb F_\pm)$ are 
    in 1-1 correspondence with functional Kan fibrations and functional signed Kan fibrations respectively. 
  \end{corollary}
  \begin{proof}
    By Remark \ref{rmkUniqueOnLength1Seq}, the vertical morphisms are uniquely defined by
    defining them on pushout sequences of length 1. By Lemma \ref{lemFaceSquares}, 
    this is uniquely done by defining the chosen lifts against (signed) horn inclusions,
    which is exactly the information the functional (signed) Kan fibrations carry. 
  \end{proof}

\subsection{Effective lifts respect degeneracy squares}
In this section, we will study what happens if we respect not only face squares, 
but also degeneracy squares. 
By Lemma \ref{lemSquaresWhoseTargetisLength1}, we only need to consider degeneracy squares 
$\delta:\sigma \to \tau$ with $\tau$ of length $1$ as in Diagram \ref{eqnPushoutSeqlength1}.
Let $s_j: \Delta^{b+1}\to \Delta^b$ be the underlying simplicial morphism of $\delta$.
It is decidable whether the index $j:\Delta^0\to \Delta^b$ factors 
through the inclusion $\Delta^n\hookrightarrow \Delta^b$. 
\begin{itemize}
  \item If $j$ does not factor through $\Delta^n$, then the pullback in Diagram \ref{eqnPushoutSeqlength1} restricts to the identity on both $\Lambda^n_m$ and $\Delta^n$ and, 
    therefore, $\Lambda^n_m$ factors through $s_j^*(T_0)$ and $\Delta^n$ factors through $s_j^*(T_1)$. 
    Therefore, for $(f,\psi)$ a vertical morphism of ${\bf RLP}(\mathbb S)$, we 
    arrive in the same situation as in Diagram \ref{eqnNoChangeSquare}, where the vertical domain
    of $\delta$ has length 1. 
    By Lemma \ref{lemFaceSquares},
    both maps $s_j^*(T_1)\to X, T_1\to X$ are defined by a pushout, 
    and by the universal property of pushouts, 
    both maps $s_j^*(T_1)\to X$ are equal, and we respect $\delta$. 
  \item If $j$ factors through $\Delta^n$, we have that $s_j^*(\Delta^n) = \Delta^{n+1}$,
    and $s_j$ restricts to $s_{j'}$ on $\Delta^{n+1}$ for some $j'$. 
    Now we have studied how the inclusion $\Lambda^n_m\hookrightarrow \Delta^n$ 
    behaves under pulling back along $s_{j'}$ in Section \ref{sec:SymEffKanFib}. 
    From that discussion, we can conclude that 
    the horn pushout sequence $\sigma$ must have one of the following forms: 
    \begin{itemize}
      \item If $j'= m$, then $\sigma$ has length 2. The first inclusion $\sigma_1$ is the pushout of a horn inclusion
        $\Lambda^n_m\hookrightarrow \Delta^n$ corresponding to either face $d_{j'}$ or $d_{j'+1}$ of 
        $\Delta^{n+1}$, while the second inclusion
        $\sigma_2$ is the pushout of the horn inclusion $\Lambda^{n+1}_{m^*}\hookrightarrow \Delta^{n+1}$
        for $m^*$ the face we did not add in $\sigma_1$. 
      \item If $j'\neq m$, then $\sigma$ has length 3. In this case
        $\sigma_1$ and $\sigma_2$ are both  the pushout of a horn
        inclusion $\Lambda^{n}_m\hookrightarrow \Delta^n$, 
        corresponding to either face $d_{j'}$ or $d_{j'+1}$ of 
        $\Delta^{n+1}$, while
        $\sigma_3$ is the pushout of the horn inclusion $\Lambda^{n+1}_{m^*}\hookrightarrow \Delta^{n+1}$
        for $m^*=m$ if $m<j'$ and $m^*=m+1$ if $m>j'$. 
    \end{itemize}  
    For each $\sigma_i$ we need a chosen horn inclusion $\iota'\in\mathcal L$ such that 
    $\sigma_i$ is the pushout of $\ell(\iota')$.
    The chosen horn inclusion for the final $\sigma_i$ 
    (in the cases above $\sigma_2$ and $\sigma_3$ respectively)
    is 
    denoted $\iota^*_\delta$, for $\iota$ as in Diagram \ref{eqnPushoutSeqlength1}.
    Note that if we respect face squares, 
    the lifts corresponding to the faces $d_{j'}, d_{j'+1}$ of $\Delta^{n+1}$, 
    will be the same as the lift against $\iota$.
\end{itemize}
We conclude that any vertical morphism $(f,\psi)$ of $\RLP(\mathbb L_\ell)$ that respects face squares 
respects the degeneracy square $\delta$ iff the lift against $\iota^*_\delta$, if it exists, 
is compatible with the lift against $\iota$, in the sense that 
for any square $\tau \to f$ restricting to $[x,y]$ on $\iota$, for any $s_{j'}$ as above and for 
\begin{equation}
  s_{j'}^*(x) = \psi_\iota(x,y) \circ s_j \circ \ell(\iota^*_\delta)
\end{equation}
(following the formulation from Equation \ref{eqnStormsForm}), we have 
\begin{equation}\label{eqnlawfscompcond}
  \psi_{\iota^*_\delta}(s_{j'}^*(x) , y \circ s_{j'}) = \psi_\iota(x,y)\circ s_{j'}
\end{equation}

Now we can define $\mathbb D_\pm$ to be the wide subcategory of $\mathbb L_\pm$ 
containing all face squares and those degeneracy squares $\delta$ 
such that whenever $\iota^*_{\delta}$ exists, it has the same sign as $\iota$. 
Note that $\iota^*_\delta$ must uniquely be the horn inclusion $\iota^*: \Lambda^{n+1}_{m^*}\hookrightarrow \Delta^{n+1}$ 
if $\ell$ is the identity on $\Lambda$. 

Now comparing the conditions in Equations 
\ref{eqndfnsymeffKancom},  \ref{eqndfnEffKanFib} and  \ref{eqnlawfscompcond},
we conclude the following:
\begin{theorem}
The vertical morphisms of $\RLP( \mathbb L_{1_\Lambda})$ and $\RLP (\mathbb D_\pm)$ are 
exactly the symmetric effective and the effective Kan fibrations respectively. 
\end{theorem}

%% file: Outlook.tex
In this final section, we will state some questions we have not addressed in this paper. 

In Section \ref{sec:Definitions}, we have introduced degenerate-preferring and 
symmetric effective Kan fibrations, 
both of which can be given the structure of effective Kan fibrations. 
In \cite{BergFaber}, the effective Kan fibrations 
are constructively shown to be closed under pushforwards. 
This is a first step in a constructive model for HoTT based on simplicial sets. 
It remains an open question whether these new notions of Kan fibrations 
are themselves also closed under pushforwards, and can form the basis for another HoTT model. 

In Section \ref{sec:Malcev}, we have extended the result that all
simplicial algebras of an algebraic theory
are Kan fibrations iff the algebraic theory is Malcev to
our other notions of Kan complexes. As we mentioned in Remark \ref{rmkMalcevCat}, 
there might be an extension of this result to Malcev categories. 
Another direction for generalization, 
lies in the Malcev structure in cubical sets with connections. 
In \cite{Tonks}, it has been shown that cubical groups with connections are Kan. 
It is an open question whether this also holds for cubical Malcev algebras with connections, 
and whether all algebraic theories whose cubical algebras with connections are Kan are Malcev algebras. 

In Section \ref{sec:LAWFS}, we have presented proofs showing that our new notions fit into lifting algebraic 
weak factorisation systems. We have used Proposition \ref{Prop:BourkeCofGenerated}, 
for which it is not yet known whether a constructive proof exists. 

In general, we do not have a good understanding of the left double category of a cofibrantly generated awfs, 
in particular, for the awfs corresponding to (symmetric) effective Kan fibrations. 
One question is whether any vertical morphism in this left double category which corresponds to an inclusion of sieves, 
is a horn pushout sequence. 
Another question in this direction is whether every pullback square of horn pushout sequences 
can be decomposed into face and degeneracy squares.

%% file: Main.bbl
\begin{thebibliography}{CCHM18}

\bibitem[BCH14]{BCH}
Marc Bezem, Thierry Coquand, and Simon Huber.
\newblock A model of type theory in cubical sets.
\newblock In {\em 19th {I}nternational {C}onference on {T}ypes for {P}roofs and
  {P}rograms}, volume~26 of {\em LIPIcs. Leibniz Int. Proc. Inform.}, pages
  107--128. Schloss Dagstuhl. Leibniz-Zent. Inform., Wadern, 2014.

\bibitem[BCP15]{BCP}
Marc Bezem, Thierry Coquand, and Erik Parmann.
\newblock Non-constructivity in {K}an simplicial sets.
\newblock In Thorsten Altenkirch, editor, {\em 13th International Conference on
  Typed Lambda Calculi and Applications (TLCA 2015)}, volume~38 of {\em Leibniz
  International Proceedings in Informatics (LIPIcs)}, pages 92--106, Dagstuhl,
  Germany, 2015. Schloss Dagstuhl--Leibniz-Zentrum fuer Informatik.

\bibitem[BF22]{BergFaber}
Benno van~den Berg and Eric Faber.
\newblock {\em Effective {K}an fibrations in simplicial sets}, volume 2321 of
  {\em Lecture Notes in Mathematics}.
\newblock Springer Cham, 2022.

\bibitem[BG16]{BourkeGarner}
John Bourke and Richard Garner.
\newblock Algebraic weak factorisation systems {I}: Accessible {AWFS}.
\newblock {\em Journal of Pure and Applied Algebra}, 220(1):108--147, jan 2016.

\bibitem[Bou23]{Bourke}
John Bourke.
\newblock An orthogonal approach to algebraic weak factorisation systems.
\newblock {\em Journal of Pure and Applied Algebra}, 227(6):107294, 2023.

\bibitem[CCHM18]{CCHM}
Cyril Cohen, Thierry Coquand, Simon Huber, and Anders M\"{o}rtberg.
\newblock Cubical type theory: a constructive interpretation of the univalence
  axiom.
\newblock In {\em 21st {I}nternational {C}onference on {T}ypes for {P}roofs and
  {P}rograms}, volume~69 of {\em LIPIcs. Leibniz Int. Proc. Inform.}, pages
  Art. No. 5, 34. Schloss Dagstuhl. Leibniz-Zent. Inform., Wadern, 2018.

\bibitem[CKP93]{CKP}
A.~Carboni, G.~M. Kelly, and M.~C. Pedicchio.
\newblock Some remarks on {M}al'tsev and {G}oursat categories.
\newblock {\em Appl. Categ. Structures}, 1(4):385--421, 1993.

\bibitem[Gar08]{GarnerSmallObjects}
Richard Garner.
\newblock Understanding the small object argument.
\newblock {\em Applied Categorical Structures}, 17(3):247--285, apr 2008.

\bibitem[Gee23]{MyThesis}
Freek Geerligs.
\newblock Symmetric effective and degenerate-preferring {K}an complexes.
\newblock Master's thesis, Utrecht University, 2023.
\newblock \url{https://studenttheses.uu.nl/handle/20.500.12932/43910}.

\bibitem[GL23]{gambinolarrea23}
Nicola Gambino and Marco~Federico Larrea.
\newblock Models of {M}artin-{L}\"of type theory from algebraic weak
  factorisation systems.
\newblock {\em J. Symb. Log.}, 88(1):242--289, 2023.

\bibitem[GS17]{GambinoSattler}
Nicola Gambino and Christian Sattler.
\newblock The {F}robenius condition, right properness, and uniform fibrations.
\newblock {\em Journal of Pure and Applied Algebra}, 221(12):3027--3068, dec
  2017.

\bibitem[Hur55]{HurewiczFiber}
Witold Hurewicz.
\newblock On the concept of fiber space.
\newblock {\em Proc. Nat. Acad. Sci. U.S.A.}, 41:956--961, 1955.

\bibitem[Joh77]{PTJ}
P.~T. Johnstone.
\newblock {\em Topos theory}, volume Vol. 10 of {\em London Mathematical
  Society Monographs}.
\newblock Academic Press [Harcourt Brace Jovanovich, Publishers], London-New
  York, 1977.

\bibitem[JP02]{KanFibrationsForMalcevAlgebras}
M.~Jibladze and T.~Pirashvili.
\newblock On {K}an fibrations for {M}altsev algebras.
\newblock {\em Georgian Math. J.}, 9(1):71--74, 2002.

\bibitem[KL21]{Voevodsky}
Krzysztof Kapulkin and Peter~LeFanu Lumsdaine.
\newblock The simplicial model of univalent foundations (after {V}oevodsky).
\newblock {\em J. Eur. Math. Soc. (JEMS)}, 23(6):2071--2126, 2021.

\bibitem[Lur23]{kerodon}
Jacob Lurie.
\newblock Kerodon.
\newblock \url{https://kerodon.net}, 2023.

\bibitem[Mal54]{Malcev}
A.~I. Mal'cev.
\newblock On the general theory of algebraic systems.
\newblock {\em Mat. Sb. N.S.}, 35(77):3--20, 1954.
\newblock Only available in Russian.

\bibitem[Moo58]{Moore}
John~C. Moore.
\newblock Semi-simplicial complexes and {P}ostnikov systems.
\newblock In {\em Symposium internacional de topolog\'{\i}a algebraica
  {I}nternational symposium on algebraic topology}, pages 232--247. Universidad
  Nacional Aut\'{o}noma de M\'{e}xico and UNESCO, Mexico City, 1958.

\bibitem[Par18]{FunctionalKan}
Erik Parmann.
\newblock {Functional Kan Simplicial Sets: Non-Constructivity of
  Exponentiation}.
\newblock In Tarmo Uustalu, editor, {\em 21st International Conference on Types
  for Proofs and Programs (TYPES 2015)}, volume~69 of {\em Leibniz
  International Proceedings in Informatics (LIPIcs)}, pages 8:1--8:25,
  Dagstuhl, Germany, 2018. Schloss Dagstuhl--Leibniz-Zentrum fuer Informatik.

\bibitem[Rie16]{RiehlCat}
Emily Riehl.
\newblock {\em Category Theory in Context}.
\newblock Aurora. Dover, 2016.

\bibitem[Swa18]{swanseparating}
Andrew Swan.
\newblock Separating path and identity types in presheaf models of univalent
  type theory, 2018.
\newblock arXiv:1808.00920.

\bibitem[Ton92]{Tonks}
A.~P. Tonks.
\newblock Cubical groups which are {K}an.
\newblock {\em J. Pure Appl. Algebra}, 81(1):83--87, 1992.

\bibitem[{Uni}13]{hottbook}
The {Univalent Foundations Program}.
\newblock {\em Homotopy Type Theory: Univalent Foundations of Mathematics}.
\newblock \url{https://homotopytypetheory.org/book}, Institute for Advanced
  Study, 2013.

\end{thebibliography}
